\documentclass[11pt]{article}
\usepackage{amsmath,amssymb,a4wide,amsthm,color,graphicx,enumerate}

\newtheorem{theorem}{Theorem}[section]
\newtheorem{corollary}[theorem]{Corollary}
\newtheorem{remark}[theorem]{Remark}

\newtheorem{proposition}[theorem]{Proposition}
\newtheorem{ex}[theorem]{Example}

\newcommand{\cb}    [1]{\ensuremath{\left  \{      #1  \right \}       }}

\newcommand{\of}    [1]{\ensuremath{\left (        #1  \right )        }}

\newcommand{\Span} {{\rm span \,}}

\newcommand{\cl}  {{\rm cl  \,}}

\newcommand{\Int} {{\rm int \,}}
\newcommand{\conv}  {{\rm conv \,}}

\newcommand{\cone}{{\rm cone\,}}

\newcommand{\ri}{{\rm ri\,}}

\newcommand{\D}{\mathcal{D}}

\renewcommand{\P}{\mathcal{P}}

\newcommand{\R}{\mathbb{R}}

\renewcommand{\subset}{\subseteq}

\newcommand{\smz}{\!\setminus\!\{0\}}

\newcommand{\dir}{^{\rm dir}}
\newcommand{\poi}{^{\rm poi}}
\newcommand{\lin}{^{\rm lin}}

\author{Andreas L\"{o}hne\thanks{Martin-Luther-Universit{\"a}t Halle-Wittenberg, Department of Mathematics, 06099 Halle (Saale), Germany andreas.loehne@mathematik.uni-halle.de}}

\title{Projection of polyhedral cones and linear vector optimization}

\begin{document}
\maketitle

\begin{abstract} Consider a polyhedral convex cone which is given by a finite number of linear inequalities. We investigate the problem to project this cone into a subspace and show that this problem is closely related to linear vector optimization: We define a cone projection problem using the data of a given linear vector optimization problem and consider the problem to determine the extreme directions and a basis of the lineality space of the projected cone $K$. The result of this problem yields a solution of the linear vector optimization problem. Analogously, the dual cone projection problem is related to the polar cone of $K$: One obtains a solution of the geometric dual linear vector optimization problem. We sketch the idea of a resulting algorithm for solving arbitrary linear vector optimization problems and provide an alternative proof of the geometric duality theorem based on duality of polytopes. 	
\medskip

\medskip


\end{abstract}

\section{Problem formulation and motivation}

Let $k,n,p$ be positive integers and let two matrices $G \in \R^{k \times n}$, $H\in \R^{k \times p}$ be given. We consider the problem to
\begin{equation}\label{P}
	\tag{P} \text{compute } K=\cb{y \in \R^p :\; G x + H y \geq 0}.
\end{equation}
A point $(x,y) \in \R^n \times \R^p$ is said to be {\em feasible} for \eqref{P} if it satisfies $G x + H y \geq 0$. A pair $(X\dir , X\lin)$ of two finite sets 
$$X\dir  =\{(x^1,y^1),\dots,(x^\alpha,y^\alpha)\},\qquad  X\lin =\{(x^{\alpha + 1},y^{\alpha +1}),\dots,(x^{\alpha + \beta}, y^{\alpha + \beta})\}$$
of feasible points
is called a {\em solution} to \eqref{P} if $\{y^{\alpha +1},\dots,y^{\alpha + \beta}\}$ is a basis of the lineality space $L:=\{y \in \R^p :\; G x + H y = 0\}$ of $K$ and $\cb{y^1,\dots, y^\alpha}$ is the set of extreme directions of $K \cap L^\bot$, where $L^\bot$ denotes the orthogonal complement of $L$.

We show in this note that every linear vector optimization problem can be expressed by a problem of type \eqref{P}. In the same manner, the dual problem to \eqref{P}, that is, 
\begin{equation}\label{D}
	\tag{P$^*$} \text{compute } K^*=\cb{w \in \R^p :\; w = -H^T u,\; G^T u = 0,\; u \geq 0}
\end{equation}
is related to the geometric dual \cite{HeyLoe08} of this linear vector optimization problem. Note that by Farkas's lemma, we have $K^* = K^\circ$, where $K^\circ:=\{w \in \R^p:\; \forall y \in K:\; w^T y \leq 0\}$ is the polar cone of $K$. Approaching linear vector optimization by a problem of type \eqref{P} has several advantages:
\begin{itemize}
	\item Even closely related to a linear vector optimization, \eqref{P} is easy to state and, in particular, free of any minimality notion. 
	\item A link between two areas is established: computation and approximation of polyhedral convex sets (see e.g. \cite{Bronstein08} for an overview) and solving and approximately solving linear vector optimization problems (see e.g. \cite{Loehne11, HamLoeRud13} and the references therein).
		\item The approach is useful for the development of (objective-space-based) algorithms for linear vector optimization problems: We show that an algorithm for problem (P) yields an algorithm for arbitrary linear vector optimization problems. On the one hand this leads to a simplification of known algorithms. On the other hand one can also cover cases which have not yet been considered in the literature: problems with empty interior of the ordering cone and problems where no minimal vertices of the image exist, see e.g. \cite{HamLoeRud13} and the references therein.
	\item In contrast to the original formulation in \cite{HeyLoe08}, geometric duality for linear vector optimization problems becomes more symmetric if it is considered in the framework of \eqref{P} and \eqref{D}. We obtain an alternative proof of the geometric duality theorem \cite{HeyLoe08}, which follows (similar to the alternative proof in \cite{Luc11}) from duality of polytopes. 
\end{itemize}

Throughout we denote by $\Int B$, $\cl B$, $\conv B$ and $\ri B$ the interior, closure, convex hull and relative interior of a set $B \subseteq \R^n$. We denote by $\cone B:= \cb{\lambda x :\; \lambda \geq 0,\; x \in \conv B}$ the convex cone generated by a set $B \subset \R^n$. If $C$ is a pointed convex cone (i.e. $C \cap -C = \cb{0}$ and $C = \cone C$), $\leq_C$ denotes the partial ordering induced by $C$, that is, $x \le_C y$ iff $y-x \in C$. A point $y \in \R^n$ is called {\em $C$-minimal} in a set $B \subseteq \R^n$ if $y \in B$ and $y \not\in B+C\smz$; $y$ is called {\em $C$-maximal} if it is $(-C)$-minimal. We denote by $0_+ B$ the recession cone (in particular, we set $0_+ \emptyset = \cb{0}$), by $L(B):= 0_+ B \cap -0_+ B$ the lineality space, and by $\dim B$ the dimension of a convex set $B\subseteq \R^n$.
By $\Span B$ we denote the linear hull of a set $B$; we define $\Span \emptyset := \cb{0}$. 
A convex set $B$ is said to be a {\em base} of a closed convex cone $C$ if $C= \cl\cone B$ and $\dim C = \dim B + 1$. For a matrix $P\in \R^{ q \times n}$ and a subset $X \subseteq \R^n$ we use the notation $P[X]:=\cb{P x :\; x \in X}$ and we set $\R^n_+:=\cb{x \in \R^n:\; x \geq 0}$. 

\section{Connection to linear vector optimization}

For positive integers $n,m,q$, let the matrices $A \in \R^{m\times n}$, $P \in \R^{q \times n}$, a vector $b \in \R^m$, and a non-trivial (i.e. $C \neq \cb{0}$) pointed polyhedral convex cone $C \subseteq \R^q$ be given. Consider the linear vector optimization problem
\begin{equation*}\label{VLP}
	\tag{VLP}	\text{min}_C Px \text{ s.t. } A x \geq b.
\end{equation*}
Its feasible set $S := \cb{x \in \R^n \mid Ax \geq b}$ is assumed to be nonempty. The set $\P:=P[S]+C$ is called the {\em upper image} of \eqref{VLP}. Let $Z \in \R^{q \times r}$ such that $C=\cb{y \in \R^q :\; Z^T y \geq 0}$. 

In the cone projection problem \eqref{P} as defined above, we set
\begin{equation}\label{eq_gh} 
	G=\begin{pmatrix} A \\ -Z^T P \\ 0\end{pmatrix} \in \R^{(m+r+1) \times n} \qquad H=\begin{pmatrix}	0 & -b \\ Z^T &0  \\ 0 &1 \end{pmatrix} \in \R^{(m+r+1) \times (q+1)}.
\end{equation}	
Then, the polyhedral convex cone $K=\cb{y \in \R^{q+1} :\; Gx + Hy \geq 0}$ in problem \eqref{P} is closely related to the upper image $\P$ of \eqref{VLP}. The following proposition shows that $\P$ is a base of the cone $K$. This base is unbounded as $C$ was assumed to be non-trivial. 

\begin{proposition}\label{p21}  
One has $K = \cl \cone (\P \times \cb{1})$.
\end{proposition}
\begin{proof}
The statement follows immediately from the facts $\cl\cone (\P \times \cb{1}) =  \cone(\P \times \cb{1}) \cup (0_+\P \times \cb{0})$ (compare \cite[Theorem 8.2]{Rockafellar72}), $\P=\{y \in \R^q:\; Ax \geq b,\; Z^T P x \leq Z^T y\}$,‚® and $0_+ \P=\{y \in \R^q:\; Ax \geq 0,\; Z^T P x \leq Z^T y\}$.  
\end{proof}

Let us turn to the dual problems. We fix some vector $c \in \ri C \times \cb{0} \subseteq \R^{q+1}$. Without loss of generality we assume (note that $c \neq 0$ since $C$ is pointed; if necessary, scale and permute coordinates)
\begin{equation}\label{ass_c}
   c_q=1.
\end{equation}
Of course, $c$ is orthogonal to the vector $ c^* := (0,\dots,0,-1) \in \R^{q+1}$. Note that the last row of $H$ is just $-c^*$ and that the last two components of a $(q+1)$-dimensional vector play a kind of extraordinary role. Throughout we use the projections
\begin{align*}
    &p\;\,:\; \R^{q+1} \to \R^q,\quad p(y) := (y_1,\dots,y_q)^T, \\
    &p^*:\; \R^{q+1} \to \R^q,\quad p^*(w) := (w_1,\dots,w_{q-1},w_{q+1})^T.
\end{align*}
A point $\bar y \in \R^q$ is said to be a {\em relatively $C$-minimal} point of $\P$ if $\bar y \in \P$, $\bar y \not\in \P+\ri C$. This notion turned out to be useful \cite{Heyde11} in order to generalize the duality results of \cite{HeyLoe08} to the case of ordering cones with empty interior. We do not use this concept in the following, because it can be replaced by minimality with respect to the ordering cone \begin{equation}\label{eq_R}
 R:=\cone\cb{p(c)},
\end{equation} 
whenever the upper image $\P$ (which involves $C$ as $\P=P[S]+C$) is considered.

\begin{proposition} Let $C$ be a pointed convex cone, $\bar c \in \ri C$ and $R:=\cone\cb{\bar c}$. The following statements are equivalent:
\begin{enumerate}[(i)]
		\item $y$ is a relatively $C$-minimal point of $\P$
		\item $y$ is an $R$-minimal point of $\P$ 
\end{enumerate} 	
\end{proposition}
\begin{proof} We have $\bar c \neq 0$ as $C$ is pointed, hence $R\smz = \ri R \subseteq \ri C$. Thus $C + \ri C = C + \ri(C + R) = C + \ri C + \ri R \subseteq C + \ri R = C + R\smz \subseteq C + \ri C$. Since $\P = \P + C$, we conclude $\P + \ri C = \P + C + \ri C = \P + C + R\smz = \P + R\smz$, which implies the statement.	
\end{proof}

Likewise to \eqref{eq_R}, we introduce an ordering cone for a dual problem as
\begin{equation*}
	R^* := \cone\cb{-p^*(c^*)}.
\end{equation*}
As we already fixed $c^* = (0,\dots,0,-1) \in \R^{q+1}$, we obtain 
$$R^*=\cb{v \in \R^q:\; v_1=\dots=v_{q-1}=0, v_q\geq 0}.$$
The (geometric) dual problem, introduced in \cite{HeyLoe08}, is
\begin{equation*}\label{VLP_star}
\tag{VLP$^*$} \text{max}_{R^*} D(u,w) \text{ s.t. } A^T u = P^T w,\; p(c)^T w = 1,\; w \in -C^\circ,\;u \geq 0,
\end{equation*}
with (linear) objective function $D:\R^m\times\R^q \to \R^q$, $D(u,w):=\of{w_1,...,w_{q-1},b^T u}^T$.
The feasible set of \eqref{VLP_star} can be expressed (see \cite{HamLoeRud13}) as
\[ T:=\cb{(u,w)\in \R^m \times \R^q :\; w = Z v,\; A^T u = P^T w,\; p(c)^T w = 1,\; u \geq 0,\; v \geq 0}.\] 
The set $\D= D[T]-R^*$ is called {\em lower image} of \eqref{VLP_star}. The bi-affine function
\[ \varphi:\; \R^q\times\R^q \to \R,\quad \varphi(y,w):=\sum_{i=1}^{q-1} y_i w_i + y_q \of{1-\sum_{i=1}^{q-1} c_i w_i} - w_q\]
is used to define the duality map
\[\Psi: 2^{\R^q} \to 2^{\R^q}, \quad\Psi( F^*):=\bigcap_{w \in  F^*}  \cb{y \in \P :\; \varphi(y,w) = 0}.\]
The following geometric duality theorem has been proven in \cite{HeyLoe08} for the case $C=\R^q_+$. An extended version similar to the following one can be found in \cite{Heyde11}.
Recall that a convex subset $F$ of a convex set $B\subseteq \R^p$ is called a {\em face} of $B$ if
\begin{equation*}
	 (y,z \in B \;\wedge\; \lambda\in (0,1) \;\wedge\; \lambda y + (1-\lambda)z \in F) \implies y,z \in F.
\end{equation*}
A face $F$ of $B$ satisfying $\emptyset \neq F \neq B$ is called {\em proper}. A face $F$ of $\P$ is said to be {\em $R$-minimal} if all points $y \in F$ are $R$-minimal in $\P$. {\em $R^*$-maximal} faces of $\D$ are defined likewise. 

\begin{theorem}[Geometric duality theorem]\label{gd}
$\Psi$ is an inclusion reversing (i.e., $F^*_1 \subseteq F^*_2 \Leftrightarrow \Psi(F^*_1)\supseteq \Psi(F^*_2)$) one-to-one map between the set of all {$R^*$-maximal} proper faces of $\D$ and the set of all {$R$-minimal} proper faces of $\P$. The inverse map is
\begin{equation*}
 \Psi^{-1}(F)=\bigcap_{y \in  F}  \cb{w \in \D:\; \varphi(y,w)=0}.
\end{equation*}
Moreover, if $F^*$ is an $R^*$-maximal proper face of $\D$, then
\[ \dim  F^* + \dim \Psi( F^*) = q-1.\]
\end{theorem}
A proof will be given in Section \ref{sec_proofs}. Using the vector $c$, which is involved in the dual problem \eqref{VLP_star}, we define the regular matrix 
\begin{equation}\label{eq_m}
	 M:=\begin{pmatrix} \!-1         &         & 0       & 0      & 0       \\
                          			 & \ddots  &         & \vdots  & \vdots           \\
									 0         &         & \!-1      & 0      & 0       \\
					                 c_1       & \dots   & c_{q-1} & 0      & \!-1      \\
					                 0         & \dots   & 0       & 1      & 0 \end{pmatrix} \in \R^{(q+1)\times(q+1)}.
\end{equation}

\begin{proposition}\label{p22} One has $K^* =  \cl \cone M(\D \times \cb{1})$.
\end{proposition}
\begin{proof}
We have $\D \times \cb{1} = \{(w_1,\dots,w_{q-1}, t, 1)^T :\; w = Z v,\; t \leq b^T u,\; A^T u = P^T w,\; p(c)^T w = 1,\; u \geq 0,\; v \geq 0\}$.  Using the assumption $c_q = 1$ in \eqref{ass_c}, we obtain $M(\D \times \cb{1}) = \{(w,t) \in \R^{q+1} :\; w = - Z v,\; t \leq b^T u,\; A^T u = - P^T w,\; p(c)^T w = -1,\; u \geq 0,\; v \geq 0\}$. Since $p(c) \in C$, it follows $p(c)^T w \leq 0$ for all $w \in C^\circ = \cb{w \in \R^q:\; w = -Z v,\; v \geq 0}$. Using the fact $\cl\cone M(\D\times\cb{1}) = \cone M(\D\times\cb{1}) \cup M(0_+\D\times\cb{0})$ (compare \eqref{eq_phi} and take into account that $M$ is just a coordinate transformation), we obtain $K^* =  \cl \cone M(\D \times \cb{1})$.
\end{proof}

Let us illustrate the geometric duality relation as well as the main idea of relating \eqref{VLP} and \eqref{VLP_star} to \eqref{P} and \eqref{D} by an example. Consider problem \eqref{VLP} with the data  
$$P = Z = \begin{pmatrix} 1 & 0 \\ 0 & 1 \end{pmatrix} \quad A=\begin{pmatrix} 2 & 1 & 1 & 0\\ 1 & 2 & 0 & 1\end{pmatrix}^T \quad b = \begin{pmatrix}1 & 1 & 0 & 0\end{pmatrix}^T,$$
and let $p(c)=(1,1)^T$. Figure \ref{f1} shows the upper image $\P$ of \eqref{VLP} and the lower image $\D$ of \eqref{VLP_star}, both are subsets of $\R^2$. In Figure \ref{f2}, we see that $\P \times \cb{1}$ is an unbounded base of a cone $K$ and $\D \times \cb{1}$ is, after an appropriate linear transformation, an unbounded base of $K^*$. 

\begin{figure}[t]
\begin{center}
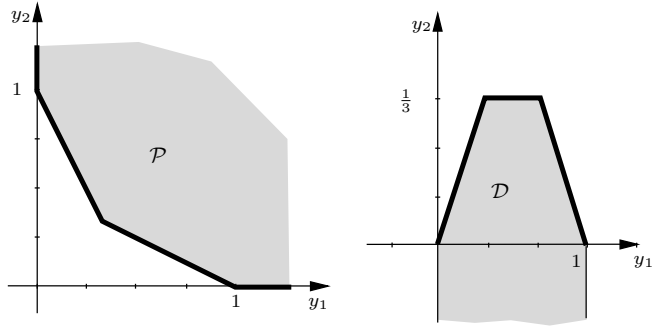
\end{center}
\caption{Geometric duality: The four vertices of $\D$ (which are $R^*$-maximal, $R^*=\cone\!\cb{(0,1)^T}$) correspond via the duality map $\Psi$ to the four facets (edges in this example) of $\P$ (which are $R$-minimal, $R=\cone\!\cb{(1,1)^T}$). Vise versa, the three vertices of $\P$ (which are $R$-minimal) correspond via the inverse duality map $\Psi^{-1}$ to the three bounded faces of $\D$ (only the bounded faces are $R^*$-maximal here).}
\label{f1}
\end{figure}

\begin{figure}[ht]
\begin{center}
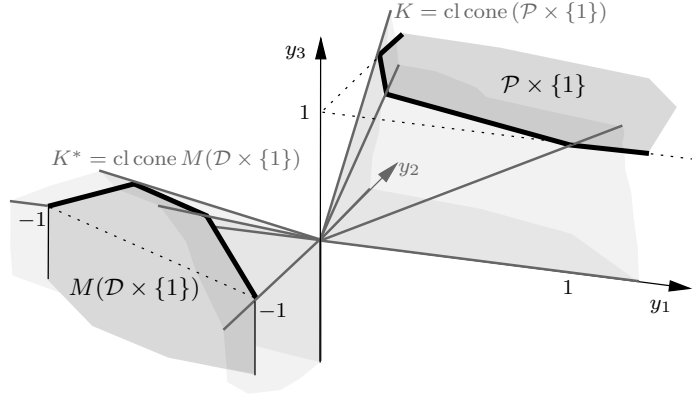
\end{center}
\caption{$K$ and its dual cone $K^\circ = K^*$ relate the upper image $\P$ of the primal problem \eqref{VLP} to the lower image $\D$ of the dual problem \eqref{VLP_star}. $\P \times \cb{1}$ is an unbounded base of $K$. Likewise, $M(\D \times \cb{1})$, where $M$ represents a coordinate transformation, see \eqref{eq_m}, is an unbounded base of $K^*$. $\P$ and a bounded base of $K$ have essentially (if $\P$'s ``faces at infinity'' are added) the same facial structure and likewise for $\D$ and a bounded base of $K^*$.}
\label{f2}
\end{figure}


Let us recall the definition of a {\em solution}  to \eqref{VLP}, compare \cite{Loehne11, HeyLoe11, HamLoeRud13}. A point $\bar x \in S$ is said to be a {\em minimizer} for \eqref{VLP} if there is no $x \in S$ such that $P x \leq_C P \bar x$, $P x \neq P \bar x$, that is, $\bar x \in S$, $P \bar x \not \in P[S] + C\smz$. To adopt this concept to directions of $S$, we consider the recession cone $0_+ S = \cb{x \in \R^n:\; A x \geq 0}$ of the (nonempty) feasible set $S$. A direction $\bar x \in \R^n \smz$ of $S$ is called a {\em minimizer} for \eqref{VLP} if $\bar x \in (0_+ S) \smz$, $P \bar x \not \in P[0_+ S] + C\smz$. Let $L(S) =\cb{x \in \R^n:\; A x = 0}$ be the lineality space of $S$. The orthogonal complement of $L(S)$ is $L(S)^\bot = \cb{x \in \R^n:\; x = A^T y,\, y \in \R^m}$. A triple $(S\poi, S\dir , S\lin) \subseteq \R^n \times \R^n\smz \times \R^n\smz$ is called {\em feasible} if $S\poi \neq \emptyset$, $S\poi \subseteq S$, $S\dir  \subseteq 0_+S$, $S\lin \subseteq L(S)$. If $(S\poi, S\dir, S\lin)$ is feasible, if the sets $S\poi$, $S\dir $, $S\lin$ are finite and if
\begin{equation}\label{eq_infatt}
\conv P[S\poi] + \cone P[S\dir ] + \Span P[S\lin] + C = P[S] + C,
\end{equation}
then $(S\poi, S\dir , S\lin)$ is called a {\em finite infimizer} for \eqref{VLP}. A finite infimizer is called a {\em solution} to \eqref{VLP} if its three components consist of minimizers only.

\begin{remark}
Note that a solution in \cite{Loehne11, HamLoeRud13} only consists of a pair $(\bar S, \bar S^h)$ rather than a triple $(S\poi, S\dir , S\lin)$. If we replace each direction $\bar x \in S\lin$ by two directions $\bar x, -\bar x \in S\dir $, we can omit the third component $S\lin$ and obtain the relation $(\bar S, \bar S^h) = (S\poi,S\dir )$.
\end{remark}

In order to relate a solution of \eqref{P} to a solution of \eqref{VLP}, for a closed convex set $B \subseteq \R^{q}$, we consider the map
\begin{equation}\label{eq_phi}
	\Phi(B):= \cl \cone (B \times \cb{1}) = \cone (B \times \cb{1}) \cup (0_+ B \times \cb{0}),
\end{equation}
where the latter equality follows from \cite[Theorem 8.2]{Rockafellar72}. 
\begin{proposition}\label{p2}
Let $B \subseteq \R^q$ be a polyhedron. Then $\Phi$ is an inclusion-invariant (i.e., $F_1 \subseteq F_2 \Leftrightarrow \Phi(F_1)\subseteq \Phi(F_2)$) one-to-one map between the set of all nonempty faces $F$ of $B$ and the set of all faces $E$ of $\Phi(B)$ with the property $E \not\subseteq 0_+B \times \cb{0}$. The inverse map is
$$ \Phi^{-1} (E) = p[E \cap (B \times\cb{1})].$$
  	For every nonempty face $F$ of $B$ one has $\dim F + 1 = \dim\Phi(F)$. 
\end{proposition}
\begin{proof}
By \eqref{eq_phi}, $\Phi$ is inclusion-invariant and enlarges the dimension by one.

If $F$ is a face of $B$, then $0_+F$ is a face of $0_+B$. Indeed, as a face $F$ of $B$ is closed and convex, $F \subseteq B$ implies $0_+F \subseteq 0_+B$. As $F$ is convex, so is $0_+F$. Choose some $x \in F$ (the case $F = \emptyset$ is obvious). Let $y,z \in 0_+B$, $\lambda \in (0,1)$ and $\lambda y +(1-\lambda)z \in 0_+F$. For all $\mu \geq 0$ we have 
$x+\mu y,\;x+\mu z \in B$ and $\lambda (x+\mu y) + (1-\lambda) (x+\mu z) \in F$. Since $F$ is a face of $B$, we obtain $x+\mu y,\; x+\mu z \in F$ for all $\mu \geq 0$ and hence $y,z \in 0_+F$.	
	
We next show that $\Phi(F)$ is a face of $\Phi(B)$, whenever $F$ is a face of $B$. Indeed, let $y,z \in \Phi(B)$, $\lambda \in (0,1)$ and $\lambda y + (1-\lambda) z \in \Phi(F)$ for a face $F$ of $B$. Using \eqref{eq_phi}, we see that there exist $\gamma,\delta > 0$ such that 
$\cb{\gamma y,\delta z} \subseteq (B \times \cb{1}) \cup (0_+B \times \cb{0})$, hence 
$(\frac{\lambda}{\gamma} + \frac{1-\lambda}{\delta})^{-1}(\frac{\lambda}{\gamma} \gamma y + \frac{1-\lambda}{\delta} \delta z) \in (F \times \cb{1}) \cup (0_+F \times \cb{0})$. Since $F$ is a face of $B$ and, as shown above, $0_+F$ is a face of $0_+B$, we conclude that $\cb{\gamma y,\delta z} \subseteq (F \times \cb{1}) \cup (0_+F \times \cb{0})$ and hence $y,z \subseteq \Phi(F)$.

Let $E$ be a face of $\Phi(B)$ such that $E \not\subseteq 0_+B \times \cb{0}$. Then there is a face $F$ of $B$ such that $\Phi(F)=E$. To prove this, we set $F:=p[E \cap (B \times\cb{1})]$. Of course, $F$ is a convex subset of $B$. To show that $F$ is a face of $B$, let $y,z \in B$, $\lambda \in (0,1)$, $\lambda y + (1-\lambda) z \in F$. We conclude that $y\times \cb{1}, z \times \cb{1} \in \Phi(B)$ and $\lambda (y\times \cb{1}) + (1-\lambda) (z\times\cb{1}) \in E$ (as $E$ is a cone). Since $E$ is a face of $\Phi(B)$, we obtain $y\times \cb{1}, z \times \cb{1} \in E$, which implies $y,z \in F$. Thus $F$ is a face of $B$.
It remains to show that $\Phi(F)= E$. We have $F \times \cb{1} \subseteq E$. This implies $\Phi(F) \subseteq E$ as $E$ is a closed cone. To show the inclusion $E \subseteq \Phi(F)$, let $y \in E \subset \Phi(B)$. If $y_{q+1} > 0$ there exists $\gamma > 0$ such that $\gamma y \in E \cap (B \times \cb{1})$, whence $y \in \Phi(F)$. Otherwise, if $y_{q+1} = 0$, we have $y \in 0_+B \times \cb{0} = 0_+(B \times \cb{1})$. By assumption we have $E \not\subseteq 0_+B \times \cb{0}$. Hence we can choose some $x \in E \cap (B \times \cb{1})$. Taking into account that $E = 0_+ E$, we obtain that $x + \mu y \subseteq E \cap (B \times \cb{1})  \subseteq \Phi(F)$ for all $\mu \geq 0$. Hence $y \in 0_+ \Phi(F) = \Phi(F)$. 
\end{proof}

It follows the main result, which shows how a solution of \eqref{VLP} can be obtained from a solution to \eqref{P}.

\begin{theorem}\label{th_sol}
	Let a linear vector optimization problem \eqref{VLP} with nonempty feasible set be given. Consider problem \eqref{P}, where the matrices $G$ and $H$ are chosen as in \eqref{eq_gh}. Let $(X\dir ,X\lin)$ be a solution to \eqref{P}. For $L_\P := \Span\!\{p(y) :\; (x,y) \in X\lin\}$, we assume
\begin{equation}\label{eq_ex}
	L_\P \cap C = \cb{0},
\end{equation}	
and define 
\begin{align*}
 S\poi &:= \cb{\frac{1}{y_{q+1}} x :\; (x,y) \in X\dir ,\; y_{q+1} > 0},\\
 S\dir &:= \cb{x :\; (x,y) \in X\dir ,\; y_{q+1} = 0,\; p(y) \not\in C + L_\P}, \\
 S\lin &:= \cb{x :\; (x,y) \in X\lin}.
\end{align*}
Then $(S\poi, S\dir, S\lin)$ is a solution to \eqref{VLP}. If \eqref{eq_ex} is violated for a solution $(X\dir ,X\lin)$ of \eqref{P}, a minimizer for \eqref{VLP} does not exist and hence \eqref{VLP} has no solution.
\end{theorem}
\begin{proof}
The $y$ components of $X\lin$ provide a basis of the lineality space $L = L(K)$ of $K$. For $G$ and $H$ as defined in \eqref{eq_gh}, we have $y_{q+1}= 0$ for all $y \in L$. Using Proposition \ref{p21} and \eqref{eq_phi}, we obtain $L_\P = L(\P)$. It is straightforward to verify that $S\poi \subseteq S$, $S\dir\subseteq 0_+S$, $S\lin \subseteq L(S)$. By assumption, we have $S \neq \emptyset$, which implies $\P \neq \emptyset$. It follows that $K \cap L^\bot$ has an extreme direction $y$ with $y_{q+1}>0$, whence $S\poi \neq \emptyset$.  
	
Let $x \in S\poi$, i.e., there is $y$ such that $(x,y) \in X\dir $ and $y_{q+1}= 1$. Setting $\bar y := p(y)$ we have $A x \geq b$ and $\bar y \geq_C P x$. We conclude that $\cone\!\cb{y} = \Phi(\cb{\bar y})$ is a one-dimensional face of $K$. Hence, $\bar y$ is a vertex of $\P$. Every vertex of $\P=P[S]+C$ is $C$-minimal, thus $Px = \bar y$ and $x$ is a minimizer for \eqref{VLP}.

Let $x \in S\dir$, i.e., there is $y$ such that $(x,y) \in X\dir$, $y_{q+1}= 0$ and $\bar y := p(y) \not\in C + L_\P$. From $y \in K$ and $y_{q+1}=0$, we conclude $\bar y \in 0_+\P$, compare 
Proposition \ref{p21} and \eqref{eq_phi}. We have $L^\bot = (L(\P) \times \cb{0})^\bot = L(\P)^\bot \times \R$. Since $y \in L^\bot$, we obtain $\bar y \in L(\P)^\bot = L_\P^\bot$. Together, we have $\bar y \in 0_+\P \cap L_\P^\bot$. Since $0_+\P \times \cb{0}$ is a face of $K$ that contains $L=L(K)$, we conclude that $(0_+\P \times \cb{0}) \cap L^\bot$ is a nonempty face of $K\cap L^\bot$. By the definition of $X\dir $, $y$ is an extreme direction of $K \cap L^\bot$. Thus $y$ is also an extreme direction of $(0_+\P \times \cb{0}) \cap L^\bot = (0_+\P \cap L_\P^\bot) \times \cb{0}$. Assume that $\bar y$ is not minimal in $0_+\P$. Then there is some $z \in 0_+ \P$ such that $\bar y-z\in C\smz$. There exist $\bar z \in L_\P$ and $\hat z \in 0_+\P \cap L_\P^\bot$ such that $z = \bar z + \hat z$. Since $\bar y - \hat z \in C + L_\P$, $0_+\P = 0_+ \P + C + L_\P$ and $\bar y - \hat z \in L_\P^\bot$, we obtain $\bar v := 2\bar y - \hat z = \bar y + (\bar y - \hat z) \in 0_+\P \cap L_\P^\bot$.  We have $\bar y = \frac{1}{2} \hat z + \frac{1}{2}\bar v$ for $z, \bar v \in 0_+\P \cap L_\P^\bot$. But $\bar y$ is an extreme direction of $0_+\P \cap L_\P^\bot$, which yields $\hat z = \mu \bar y$ for some $\mu \in \R$. We have $\bar y \neq \hat z$, since otherwise $-\bar z = \bar y - z \in C\smz \cap L_\P$, which contradicts \eqref{eq_ex}. Thus $\mu \neq 1$. Assuming that $\mu < 1$, we obtain $\bar y = (1-\mu)^{-1} (\bar y - \hat z) \in C + L_\P$, which contradicts the definition of $S\dir$. Therefore the case $\mu > 1$ remains. We obtain $\bar y = (1-\mu)^{-1} (\bar y - \hat z) \in -C + L_\P$, which implies $y \in -K + L = -K$. But $y \in K \cap L^\bot$, whence $y \in L \cap L^\bot = \cb{0}$ and thus $0 = \bar y \in C+ L_\P$, which contradicts the definition of $S\dir$. 

Let $x \in S\lin$, i.e., there is $y$ such that $(x,y) \in X\lin$. We have $\bar y:= p(y) \in L_\P = L(\P) \subseteq 0_+\P$. Assume that $\bar y$ is not minimal in $0_+\P$, i.e., there exists $\bar z \in 0_+\P$ such that $\bar y - \bar z \in C\smz \subseteq 0_+\P$. Moreover, we have $\bar y - \bar z = L(\P) - 0_+\P = -0_+\P$. We conclude that $\bar y - \bar z \in L(\P) \cap C\smz$, which contradicts \eqref{eq_ex}.  
	
To verify \eqref{eq_infatt} it remains to show the inclusion $\supseteq$. Let $\bar y \in P[S]+ C$ be given. Then, $y := (\bar y,1)^T \in K$ can be expressed by a solution $(X\dir ,X\lin)$ with $X\dir =\{(x^1,y^1),\dots,(x^\alpha,y^\alpha)\}$, $X\lin=\{(x^{\alpha+1},y^{\alpha+1}),\dots,(x^{\alpha+\beta},y^{\alpha+\beta})\}$, and appropriate $\lambda_1,\dots,\lambda_\alpha \in \R_+$, $\lambda_{\alpha+1},\dots,\lambda_{\alpha+\beta} \in \R$ as
$$ y = \sum_{i=1}^{\alpha+\beta} \lambda_i y^i =
 \sum_{i \in I_1} \lambda_i y^i_{q+1} \frac{y^i}{y^i_{q+1}} + \sum_{i\in I_2 \cup I_3 \cup I_4} \lambda_i y^i,$$
where we consider the disjoint index sets
\begin{align*}
I_1 &= \cb{i \in \cb{1,\dots, \alpha}:\; y^i_{q+1} > 0},\\ 
I_2 &= \cb{i \in \cb{1,\dots, \alpha}:\; y^i_{q+1} = 0,\; y^i \not\in (C + L_\P) \times\cb{0}},\\
I_3 &= \cb{i \in \cb{1,\dots, \alpha}:\; y^i \in (C+ L_\P) \times\cb{0}},\\
I_4 &= \cb{\alpha+1,\dots, \alpha + \beta}.
\end{align*} 	
For $i \in I_1$, we have $(y^i_{q+1})^{-1} y^i \in (P[S\poi] + C) \times \cb{1}$. Moreover, we have $y^i \in (P[S\dir] + C) \times \cb{0}$ for $i \in I_2$, $y^i \in (C + P[S\lin]) \times \cb{0}$ for $i \in I_3$ (note that $L_\P = P[S\lin]$) and $y^i \in P[S\lin] \times \cb{0}$ for $i\in I_4$. We conclude that $\sum_{i\in I_1} \lambda_i y^i_{q+1} = y_{q+1} = 1$. Together we obtain
$y \in \conv P[S\poi] + \cone P[S\dir] + \Span P[S\lin] + C$, which completes the proof of \eqref{eq_infatt}. 

Finally, assume that $\bar x$ is a minimizer of \eqref{VLP}, but \eqref{eq_ex} is not satisfied. 
We can choose some $\bar y \in C\smz \cap L(\P)$. If $\bar x$ is a point, we have $\hat y := P \bar x - \bar y \in \P$. There exists $\hat x \in S$ such that $\hat y - P\hat x \in C$. Thus $P \bar x - P \hat x = \hat y - P \hat x +\bar y \in C + C\smz = C\smz$. If $\bar x$ is a direction, we obtain $\hat y := P \bar x - \bar y \in 0_+\P = P[0_+ S] + C$. There exists $\hat x \in 0_+S$ such that $\hat y - P \hat x \in C$. As above, we conclude $P \bar x - P \hat x = C\smz$. In both cases, this contradicts the definition of a minimizer.
\end{proof}

A {\em solution} to \eqref{VLP_star} is defined as follows, compare \cite{Loehne11, HeyLoe11, HamLoeRud13} for a special case. A point $(\bar u, \bar w) \in T$ is said to be a {\em maximizer} for \eqref{VLP_star} if there is no $(u,w) \in T$ such that $D(\bar u,\bar w) \leq_{R^*} D(u,w)$, $D(\bar u,\bar w) \neq D(u,w)$, that is, $(\bar u,\bar w) \in T$, $D(\bar u, \bar w) \not \in D[T] - R^*\smz$. A direction $(\bar u, \bar w) \in \R^{m+q}\smz$ of $T$ is called a {\em maximizer} for \eqref{VLP_star} if $(\bar u,\bar w) \in (0_+ T)\smz$, $D(\bar u,\bar w) \not \in D[0_+ T] - R^*\smz$. A triple $(T\poi, T\dir , T\lin) \subseteq \R^{m+q} \times \R^{m+q}\smz \times \R^{m+q}\smz$ is called {\em feasible} if $T\poi \neq \emptyset$, $T\poi \subseteq T$, $T\dir  \subseteq 0_+T$, $T\lin \subseteq L(T)$. If $(T\poi, T\dir, T\lin)$ is feasible, if the sets $T\poi$, $T\dir $, $T\lin$ are finite and if
\begin{equation*}\label{eq_supatt}
\conv D[T\poi] + \cone D[T\dir ] + \Span D[T\lin] - R^* = D[T] - R^*,
\end{equation*}
then $(T\poi, T\dir , T\lin)$ is called a {\em finite supremizer} for \eqref{VLP_star}. A finite supremizer is called a {\em solution} to \eqref{VLP_star} if its three components consist of maximizers only.

Let us introduce a solution concept for the dual cone projection problem \eqref{D} in order to relate it to a solution $(T\poi,T\dir,T\lin)$ of \eqref{VLP_star}. A point $(u,w) \in \R^m \times \R^p$ is said to be feasible for \eqref{D} if it satisfies $w = -H^T u,\; G^T u = 0,\; u \geq 0$. A pair $(U\dir, U\lin)$ of two finite sets $U\dir =\cb{(u^1,w^1),\dots,(u^\gamma,w^\gamma)}$ and $U\lin =\cb{(u^{\gamma + 1},w^{\gamma +1}),\dots,(u^{\gamma + \delta}, w^{\gamma + \delta})}$ of feasible points
is called a {\em solution} to \eqref{D} if $\cb{w^{\gamma +1},\dots,w^{\gamma + \delta}}$ is a basis of the lineality space $L(K^*)$ of $K^*$ and $\cb{w^1,\dots, w^\gamma}$ is the set of extreme directions of $K^* \cap L(K^*)^\bot$.

\begin{theorem}\label{th_sol_d}
Assume that the dual linear vector optimization problem \eqref{VLP_star} has a non\-empty feasible set.  Let $(U\dir ,U\lin)$ be a solution to the dual cone projection problem \eqref{D}. For $L_\D := \Span\!\{p(M^{-1}w) :\; (u,w) \in U\lin\}$ assume that
\begin{equation}\label{eq_ex_d}
	L_\D \cap R^* = \cb{0},
\end{equation}	
and define 
\begin{align*}
 T\poi &:= \cb{\frac{-1}{c^T w} (u,M^{-1}w) :\; (u,w) \in U\dir ,\; c^T w < 0},\\
 T\dir &:= \cb{(u,M^{-1}w) :\; (u,w) \in U\dir ,\; c^T w = 0,\; p^*(w) \not\in R^* + L_\D}, \\
 T\lin &:= \cb{(u,M^{-1}w) :\; (u,w) \in U\lin}.
\end{align*}
Then $(T\poi, T\dir, T\lin)$ is a solution to \eqref{VLP_star}. If \eqref{eq_ex_d} is violated for a solution $(U\dir ,U\lin)$ of \eqref{D}, then there does not exist a maximizer for \eqref{VLP_star} and hence \eqref{VLP_star} has no solution.
\end{theorem}
\begin{proof}
The proof is analogous to the proof of Theorem \ref{th_sol} if we consider  
$$\bar U\dir = \cb{(u,M^{-1} w) : (u,w) \in U\dir} \quad\text{ and }\quad \bar U\lin = \cb{(u,M^{-1} w) : (u,w) \in U\lin},$$	
compare Propositions \ref{p21} and \ref{p22}. We have $M^{-1} w = (-w_1,\dots,-w_{q-1},w_{q+1}, -c^T w)^T$ and we note that $-R^* \subset 0_+ \D$ is the replacement for $C \subseteq 0_+\P$ as we consider maximization instead of minimization. Note further that $p(M^{-1} w) \not\in -R^*  + L_\D$ can be written as $p^*(w) \not\in R^*  + L_\D$.	
\end{proof}

Note that, in order to enhance the symmetry between \eqref{VLP} and \eqref{VLP_star}, the ordering cone $R^*$ in the dual problem could be replaced by a pointed polyhedral convex cone $C^*$ satisfying $R^* \subseteq C^* \subseteq - 0_+ \D$.

\section{Consequences for linear vector optimization algorithms}

Using the results of the last section we want to propose an algorithm to solve linear vector optimization problems, which is based on the computation of the vertices of a polytope. By Theorem \ref{th_sol}, \eqref{VLP} can be solved by determining the extreme directions of $K$ as well as a basis of the lineality space $L(K)$ of $K$. A solution of the dual problem \eqref{VLP_star} can be obtained likewise by Theorem \ref{th_sol_d}. We start with some facts about the facial structure of the polyhedral convex cones $K$ and $K^*$.


Let $L(K)$ and $L(K^*)$ be the lineality spaces of $K$ and $K^*$, respectively. Consider $\hat K := K \cap L(K)^\bot$ and $\hat K^* := K^* \cap L(K^*)^\bot$. Setting $V:=L(K)^\bot \cap L(K^*)^\bot$, we have 
$\hat K \subseteq V$ and $\hat K^* \subseteq V$. Both $\hat K$ and $\hat K^*$ are pointed. From $L(K) \subseteq K$ we conclude $K^* = K^\circ \subseteq L(K)^\bot$ and hence $(K^* + L(K))\cap L(K)^\bot = K^*$. This implies $(\hat K)^\circ \cap V = (K \cap L(K)^\bot)^\circ \cap V = (K^\circ + L(K))\cap L(K)^\bot \cap L(K^*)^\bot = K^* + L(K^*)^\bot = \hat K^*$. Likewise we have $(\hat K^*)^\circ \cap V = \hat K$. As $\hat K$ and $\hat K^*$ are pointed convex cones that are polar to each other relative to $V$, we conclude that both have nonempty interior relative to $V$.
	
Let $\xi \in \ri \hat K$ and $\eta \in \ri \hat K^*$ such that $\xi^T \eta = -1$. Then, $B:=\{y \in \hat K :\; \eta^T y = -1\}$ and $B^*:=\{w \in \hat K^* :\; \xi^T w = -1\}$ provide a bounded base of $\hat K$ and $\hat K^*$, respectively. Applying Proposition \ref{p2} and taking into account an appropriate coordinate transformation, we obtain an inclusion-invariant one-to-one map between the nonempty faces $F$ of $B$ and the nonempty faces $\hat F$ of $\hat K$ with the property $\hat F \neq \cb{0}$, and likewise for $B^*$ and $\hat K^*$. Using appropriate coordinates, $B_\xi := B - \{\xi\}$ and $B_\eta^* := B^* - \{\eta\}$ are mutually polar polytopes in $\R^{p-1}$, that is $B_\xi=\{y \in \R^{p-1} :\; \forall w \in B_\eta^*:\; w^T y \leq 1\}$ and $B_\eta^*=\{w \in \R^{p-1} :\; \forall y \in B_\xi:\; y^T w \leq 1\}$. Indeed let $y \in B$ and $w \in B^*$. Then $ (y-\xi)^T (w-\eta) = y^T w - \xi^T w - \eta^T y + \xi^T \eta = y^T w + 1 \leq 1$ if and only if $y^T w \leq 0$. 

A resulting algorithm to solve \eqref{VLP} and \eqref{VLP_star} can be outlined as follows:
\begin{enumerate}[(i)]
	\item Compute $L(K)$, $L(K^*)$ as well as two finite sets $E \subseteq \hat K$, $E^* \subseteq \hat K^*$ such that $\dim \cone E = \dim \cone E^* = \dim V$. Determine $\xi \in \ri \hat K$ and $\eta \in \ri \hat K^*$ such that $\xi^T \eta = -1$. 
	\item Consider an appropriate coordinate transformation and a suitable subspace such that $B$ and $B^*$ are polytopes with nonempty interior. Transform $(\cone E)^\circ$ and $(\cone E^*)^\circ$ in the same way as $K$ and $K^*$ in order to obtain polytopes $Q \supseteq B$ and $Q^* \supseteq B^*$ (in contrast to $B$, $B^*$ the vertices of $Q$, $Q^*$ are known or can be easily obtained).
	\item Proceed with an outer approximation algorithm to compute the vertices of $B$ and $B^*$. To this end adapt (simplify) Benson's algorithm \cite{Benson98} to polytopes. See also the references in \cite[Section 8.3]{Bronstein08} for similar methods in the field of approximation of convex bodies.
	\item Compute solutions to \eqref{VLP} and \eqref{VLP_star} by Theorems \ref{th_sol} and \ref{th_sol_d}
\end{enumerate}

In comparison with objective-space-based algorithms (Benson type algorithms) for linear vector optimization problems (see e.g. \cite{HamLoeRud13} and the references therein), the advantages of the new method are as follows:
\begin{itemize}
	\item It is not necessary to distinguish between homogeneous and inhomogeneous problems to treat unbounded problems (compare \cite{Loehne11,HamLoeRud13}).
	\item In contrast to the methods in the literature (see e.g. \cite{HamLoeRud13} for an overview), this approach covers also the case where the upper image $\P$ of \eqref{VLP} has no vertex. Moreover, the ordering cone $C$ is allowed to have an empty interior. 
	\item There is no formal difference between primal and dual algorithms as the same idea can be applied to the polar cone $K^*$. However, note that the dual algorithm uses different (transformed) data and can therefore be better or worse than the primal algorithm dependent on the problem instance.
\end{itemize}
Detailed algorithms as well as numerical results will be presented in a forthcoming paper. Finally we summarize the above considerations in order to apply them in the next section.

\begin{proposition}\label{pgamma}
	Let $K \subseteq \R^p$ be a polyhedral convex cone. The map $\Gamma$ defined by
	$$ \Gamma(F):= \bigcap_{y \in F} \cb{w \in K^* :\; w^T y = 0}$$
	provides an inclusion-reversing one-to-one map between the nonempty faces $F$ of $K$ and the nonempty faces $F^*$ of $K^*$. The inverse map is
$$ \Gamma^{-1}(F^*) := \bigcap_{w \in F^*} \cb{y \in K :\; y^T w = 0}.$$
and for all nonempty faces $F$ of $K$ one has $\dim F + \dim \Gamma(F) = p$.	
\end{proposition}
\begin{proof}
	Taking into account the considerations above, we obtain this result from duality of polytopes \cite{Gruenbaum03}. Note further that $\Delta$ defined by $\Delta(\hat F):= \hat F + L(K)$ provides an inclusion invariant one-to-one map between set of all faces of $\hat K:= K \cap L(K)^\bot$ and set of all faces of $K$, where $\Delta^{-1}(F) = F \cap L(K)^\bot$; likewise for $K^*$.   
\end{proof}

	

\section{Alternative proof of the geometric duality theorem} \label{sec_proofs}

Using the results of the previous sections and a few additional components we obtain an alternative proof of the geometric duality theorem. Note that some basic ideas of this proof can already be found in \cite{Luc11}, where a parametric dual problem is introduced, polarity between a polyhedral set and the epigraph of its support function is utilized to prove duality assertions, and geometric duality (for the special case $C=\R^q_+$) is shown to be a consequence. 

We will see that the duality map $\Psi$ in Theorem \ref{gd} can be expressed as
\begin{equation}\label{eq_dm}
\Psi(F^*) = (\Phi^{-1} \circ \Gamma^{-1}\circ M\circ\Phi)(F^*),
\end{equation}
compare Propositions \ref{p2} and \ref{pgamma}, and \eqref{eq_m}. 

\begin{proposition}\label{pr_41}
The following statements are equivalent:
	\begin{enumerate}[(i)]
		\item $F$ is an $(R\times \cb{0})$-minimal face of $K$,
		\item There exists $w \in \Gamma(F)$ such that $c^T w < 0$.
	\end{enumerate}	
\end{proposition}
\begin{proof}
Note first that $c \in 0_+\P \times\cb{0} \subseteq K$. Since $R\times\cb{0} = \cone\!\cb{c}$, (i) is equivalent to 
\begin{equation}\label{eq_41}
   \forall \varepsilon > 0,\; \forall y \in F:\; y - \varepsilon c \not\in K.
\end{equation}
\eqref{eq_41} $\Rightarrow$ (ii): There exists $w \in \Gamma(F)$ such that $w^T c < 0$ since otherwise,  by Proposition \ref{pgamma}, we obtain $c \in F$, which contradicts \eqref{eq_41}. (ii) $\Rightarrow$ \eqref{eq_41}: Let $y \in F$ and $\varepsilon>0$. Then $w^T(y-\varepsilon c) = -\varepsilon c^T w > 0$, i.e., $y-\varepsilon c \not\in K$.	
\end{proof}

\begin{proposition}\label{pr_42}
The following statements are equivalent:
	\begin{enumerate}[(i)]
		\item $F^*$ is an $(M(R^*\times \cb{0}))$-maximal face of $K^*$,
		\item There exists $y \in \Gamma^{-1}(F^*)$ such that ${c^*}^T y < 0$.
	\end{enumerate}	
\end{proposition}
\begin{proof}
We have $M(R^* \times\cb{0}) = \cone (M(\{-p^*(c^*)\} \times \{0\})) = \cone \cb{-c^*}$. Thus (i) is equivalent to 
$$ \forall \varepsilon > 0,\; \forall w \in F^*:\; w - \varepsilon c^* \not\in K^*.$$
The remaining arguments are analogous to those in Proposition \ref{pr_41}.	
\end{proof}

\begin{corollary}\label{cor_42}
The following statements are equivalent:	
	\begin{enumerate}[(i)]		
		\item $F^*$ is an $R^*$-maximal face of $\D$.
		\item $(\Phi^{-1} \circ \Gamma^{-1} \circ M\circ\Phi)(F^*)$ is an $R$-minimal face of $\P$,
	\end{enumerate}
\end{corollary}
\begin{proof}
	Using Proposition \ref{p2}, we see that (i) is equivalent to $\tilde F^* := (M \circ \Phi)(F^*)$ being a $(M(R^* \times \cb{0}))$-maximal face of $K^*$ such that $\tilde F^* \not\subseteq M(0_+\D \times \cb{0})$. This is equivalent to (ii) in Proposition \ref{pr_41} and (i) in Proposition \ref{pr_42} for $F^* = \Gamma(F)$. Hence this is equivalent to (i) in Proposition \ref{pr_41} and (ii) in Proposition \ref{pr_42}. This means that $\tilde F:= (\Gamma^{-1} \circ M \circ \Phi)(F^*)$ is an $(R\times\cb{0})$-minimal face of $K$ and we have $\tilde F \not\subseteq 0_+\P\times \cb{0}$. By Proposition \ref{p2} this is equivalent to (ii).	
\end{proof}

 Now, \eqref{eq_dm} can be verified by a straightforward calculation. Combining the results of Propositions \ref{p2} and \ref{pgamma} and Corollary \ref{cor_42} we complete the proof of the geometric duality theorem.

\bibliographystyle{abbrv}
\bibliography{database}

\end{document}